\documentclass{mcom-l}
\newcommand{\Poly}{\mathbb{P}}
\newcommand{\bR}{{\mathbb R}}
\newcommand{\W}{{\mathcal W}}
\newcommand{\B}{{\mathcal B}}
\newcommand{\R}{{\mathbb R}}
\newcommand{\N}{{\mathbb N}}
\newcounter{bean}
\newenvironment{romenum}{\begin{list}{{(\roman{bean})}}
{\usecounter{bean}}}{\end{list}}
\copyrightinfo{2006}{American Mathematical Society}
\newtheorem{theorem}{Theorem}[section]
\newtheorem{lemma}[theorem]{Lemma}

\theoremstyle{definition}

\theoremstyle{remark}
\newtheorem{remark}[theorem]{Remark}

\numberwithin{equation}{section}
\begin{document}

\title[DGM FOR VOLETRRA INTEGRO-DIFFERENTIAL EQUATIONS]
{A Superconvergent  discontinuous Galerkin method for Volterra integro-differential equations, smooth and non-smooth kernels}

\author{Kassem Mustapha}
\address{Department of Mathematics and Statistics,
King Fahd University of Petroleum and Minerals, Dhahran, 31261, Saudi Arabia.} \email{kassem@kfupm.edu.sa}
\thanks{Support of the KFUPM through the project SB101020 is gratefully acknowledged.}
\maketitle
\begin{abstract}
We study the numerical solution for Volerra integro-differential equations with smooth and non-smooth kernels. We use a $h$-version
discontinuous Galerkin (DG) method and derive nodal error bounds that are explicit in the parameters of interest. In the case of non-smooth
kernel, it is justified that the start-up singularities can be resolved at superconvergence rates by using non-uniformly graded meshes. Our
theoretical results are numerically validated in a sample of test problems.
\end{abstract}


\section{Introduction}
 In this paper, we study the
discontinuous Galerkin (DG) for a nonlocal time dependent Volterra integro-differential equation of the form
\begin{equation}\label{eq: VIE}
u'(t) +a(t)u(t) +\B u(t) = f(t),~~ 0<t<T~~{\rm with}~~u(0)=u_0,
\end{equation} where  $\B$ is the
Volterra operator:
\begin{equation}
\label{eq:kernel} \B u(t)=\int_0^t \beta(t,s) u(s)\,ds,
\end{equation}
such that,
\begin{equation}\label{eq: VIE1}
\beta(t,s)= (t-s)^{\alpha-1} b(s)~~ \text{ ~~for~all~ $0<s<t\le T$}
\end{equation}
with either  $\alpha \in (0,1)$ (weakly singular kernel) or $\alpha \in   \N_0:=\{1,2,3,\cdots\}$ (smooth kernel). Here $a$, $b$ and $f$ are
continuous real valued functions on $[0,T]$. We assume that there exist $\mu_*
>0$ such that $ a(t) \ge \mu_*$ for all $t\in [0,T]$. As a consequence of this and the continuity assumptions on the functions $a$ and $b$\,;
there exist $\mu_*\,,\mu^* >0$ such that \begin{equation}\label{eq: mu*} \mu_*\le a(t) \le \mu^*\quad{\rm and}\quad |b(t)|\le \mu^*\quad {\rm
for~all}~~~t\in [0,T]\,.\end{equation}

For any $u_0 \in R$, problem (\ref{eq: VIE}) has a unique solution $u$ which is continuously differentiable, see for example \cite{Brunner04}.
However for $\alpha \in (0,1)$, even if the functions $a$, $b$ and $f$ in (\ref{eq: VIE})--\eqref{eq: VIE1} are smooth, the second derivative of
$u$ is not bounded at $t=0$ (see \cite{BrunnerPedasVainikko01} and related references therein), and behaves like $|u''(t)|\le Ct^{ \alpha-1}$.
The singular behavior of $u$ near $t=0$ may lead to suboptimal convergence rates if we work with quasi-uniform time meshes. To overcome this
problem, we employ a family of non-uniform meshes, where the time-steps are concentrated near $t=0$.

Various numerical methods had been studied for problem \eqref{eq: VIE}. For instance, collocation methods for (\ref{eq: VIE}) with a weakly
singular kernel  were investigated by many authors where an $O(k^{p+1})$ ($k$ is the maximum time-step size and $p$ is the degree of the
approximate solution) global convergence rate had been achieved using a non-uniform graded mesh of the form (\ref{eq: standard tn}), see for
example \cite{Brunner04,BrunnerPedasVainikko01,Tang1993} and references therein.  Spectral methods and the corresponding error analysis were
provided in \cite{16, 28}  assuming that $\alpha=1$ and the solution $u$ of (\ref{eq: VIE}) is smooth. However, for $0<\alpha<1$ (that is, the
kernel is weakly singular), the spectral collocation method were recently studied in \cite{29} where the convergence analysis was carried out
assuming again  that the solution $u$  is smooth. For other numerical tools, refer to \cite{29} and references therein.

In the present paper we shall study the nodal error analysis for the DG time-stepping method  (with a fixed approximation order) applied to
problem (\ref{eq: VIE}). Indeed, the DG time-stepping method for \eqref{eq: VIE} when $\alpha \in (0,1)$ has been introduced
in~\cite{BrunnerSchoetzau06}, where a uniform optimal $O(k^{p+1})$  convergence rate had been shown assuming that $u$ is sufficiently regular.
In this work, we show that a faster convergence than~$O(k^{p+1})$ is possible at the nodal points. For a weakly singular kernel ($\alpha \in
(0,1)$), we prove that by using non-uniformly refined time-steps, start-up singularities near $t=0$ can be resolved at
$O(k^{\min\{p,\alpha+1\}+p+1})$ superconvergence rates\,. Such convergence rates can not be obtained by using the approach given in
\cite{BrunnerSchoetzau06}. Very briefly, our proof technique will be carried out in two steps; deriving first the global convergence results of
the DG method for the dual problem of (\ref{eq: VIE}) (which is essential for the nodal error but irrelevant for the global error estimates),
see Theorem \ref{v2thm: U-Pi u}. Then, we use these results with the orthogonal property of the DG scheme for (\ref{eq: VIE}) very appropriately
(see (\ref{v2eq: duality trick}) and Theorem \ref{thm: U^n-u(t_n)}) to achieve nodal superconvergence estimates. For smooth kernels ($\alpha\in
\N_0$), we appropriately modify our earlier analyses to show nodal superconvergence rates  of order $O(k^{2p+1})$ assuming that the functions
 $a$, $b$ and  $f$ are sufficiently regular (see Theorem \ref{thm: U^n-u(t_n) smooth}).

The origins of the DG methods can be traced back to the seventies where they had been proposed as variational methods for numerically solving
initial-value problems and transport problems~\cite{LesaintRaviart74,ReedHill73, DelfourHagerTrochu81,Estep95,Johnson88} and the references
therein. In the eighties, DG time-stepping methods were successfully applied to parabolic problems; see for example,
\cite{ErikssonJohnsonThomee1985}, where a nodal $O(k^{2p+1})$ superconvergence rate had been proved.  Subsequently, in
\cite{LarssonThomeeWahlbin98}, a piecewise linear time-stepping DG method had been proposed and studied for a parabolic integro-differential
equation:
\begin{equation}\label{eq: VIE parabolic}
u_t +A u +\B  \widetilde A u = f ~~~ {\rm in}~~(0,T]\times\Omega~~{\rm with}~~u(0)=v(x)~~{\rm on}~~\Omega~~{\rm for}~\alpha\in (0,1),
\end{equation}
where $\Omega \subset \bR^d$ is a bounded convex domain, $A$ is a linear self-adjoint, positive-definite operator (spatial), with compact
inverse, defined in $D(A)$, and where $A$ dominates the spatial operator $\widetilde A$. A nodal $O(k^3)$ superconvergence rate had been derived
 assuming that $b(s)=1$ in (\ref{eq: VIE1}), where  the error analysis there was based on the fact that on each time
interval, the DG solution takes its maximum values on one of the end points. However, this is not true in the case of DG methods of higher order
p. The high order time-stepping DG for (\ref{eq: VIE parabolic}) was investigated in \cite{MustaphaBrunnerMustaphaSchoetzau} where a global
optimal $O(k^{p+1})$ convergence rate had been proved, assuming that the mesh is non-uniformly graded. (For other numerical methods for
\eqref{eq: VIE parabolic}, see \cite{Mustapha2008, Mustapha2010, Pani2010, Pani2011} and related references therein.)  Indeed, our convergence
analysis can in principle be extended to cover the nodal error estimates from the DG time-stepping method of order $p$, applied to (\ref{eq: VIE
parabolic}).

The outline of the paper is as follows. In Section~\ref{sec:numerical-method}, we introduce the DG time-stepping method with a fixed
approximation degree $p$ (typically low) on non-uniformly refined time-steps with $p\ge 1$. In Section~\ref{sec:time-error}, we give a global
formulation of the DG scheme, introduce our projection operator,
 and also provide some technical lemmas. In Section~\ref{sec:dual}, we define the dual of the problem (\ref{eq: VIE}) and  then derive the
error estimates from the discretization by the DG method when $\alpha\in (0,1)$; see Theorem \ref{v2thm: U-Pi u}. In
Section~\ref{sec:nodal-error}, we prove our main nodal error bounds. For $\alpha \in (0,1)$, an error $|U_-^n-u(t_n)|$ of order
$O(k^{\min\{p,\alpha+1\}+p+1})$ (i.e., superconvergent of order $k^3$ for $p=1$ and $k^{p+2+\alpha}$ for $p\ge 2$) has been shown provided that
the solution $u$ of (\ref{eq: VIE}) satisfies ~\eqref{eq:countable-regularity v1} and the mesh grading parameter $\gamma>(p+1)/\sigma$; see
Theorem \ref{thm: U^n-u(t_n)}. In Section \ref{sec:smooth kernel}, we consider the case $\alpha \in   \N_0$ (in (\ref{eq: VIE1})) and thus the
kernel is smooth. We show a nodal error of order $O(k^{2p+1})$  (over a uniform mesh) assuming that the solution $u$ of (\ref{eq: VIE}) is
sufficiently regular, refer to Theorem \ref{thm: U^n-u(t_n) smooth}.  We present a series of numerical examples to validate our theoretical
results in Section~\ref{sec:numerics}.
\section{Discontinuous Galerkin time-stepping}\label{sec:numerical-method}
To describe the DG  method, we introduce a (possibly non-uniform) partition of the time interval $[0,T]$ given by the points
\begin{equation}\label{eq: tn mesh}
0=t_0<t_1<\cdots<t_N=T.
\end{equation}
We set $I_n=(t_{n-1},t_n)$ and $k_n=t_n-t_{n-1}$ for $1\le n\le N$. The maximum step-size is defined as $k=\max_{1\le n\le N}k_n$. We now
introduce the discontinuous finite element space
\begin{equation}
\label{eq:FE-space} \W_p=\left\{\,v:J_N \to \R\,:\, v|_{I_n}\in\Poly_{p},\ 1\le n\le N\right\},
\end{equation}
where $J_N=\cup_{n=1}^N I_n$, and  $\Poly_{p}$ denotes the space of polynomials of degree $\le p$ where $p$ is a positive integer $\ge 1$. We
denote the left-hand limit, right-hand limit and jump at $t_n$ by $ v_-^n=v(t_n^-)$, $v_+^n=v(t_n^+)$ and $[v]^n=v^n_+-v_-^n$, respectively.

The DG approximation $U\in \W_p$  is now obtained as follows: Given $U(t)$ for $t\in I_j$ with $1\le j\le n-1$, the approximation $U\in
\Poly_{p}$ on the next time-step~$I_n$ is determined by requesting that
\begin{equation}
\label{eq: DG step}
\begin{split}
& U^{n-1}_+X^{n-1}_+
    +\int_{t_{n-1}}^{t_n}
    \Bigl[U'+a(t)U(t)+\B U(t)\Bigr]X\,dt
=U_-^{n-1}X^{n-1}_++\int_{t_{n-1}}^{t_n}f\,X\,dt
\end{split}
\end{equation}
for all test functions $X\in\Poly_{p}$.  This time-stepping procedure starts from  $U^0_-=u_0$, and after $N$~steps it yields the approximate
solution~$U\in \W_p$ for $t \in J_N$.
\begin{remark}\label{remark: piecewise onstant case}
For the piecewise-constant case~$p=0$, since $U'(t)=0$ and $U(t)=U_-^n=U^{n-1}_+=:{\bf U}^n$ for~$t\in I_n$, the DG method~\eqref{eq: DG step}
amounts to a generalized backward-Euler scheme
\begin{multline*}
\frac{{\bf U}^n-{\bf U}^{n-1}}{k_n}+
    U^n\frac{1}{k_n}\int_{t_{n-1}}^{t_n}a(t)\,dt +\omega_{nn}k_n {\bf U}^n\\
    =\frac{1}{k_n}\int_{t_{n-1}}^{t_n}f(t)\,dt-\frac{1}{k_n}\int_{t_{n-1}}^{t_n}\sum_{j=1}^{n-1}{\bf U}^j\int_{t_{j-1}}^{\min(t,t_j)}
    (t-s)^{\alpha-1}b(s)\,ds\,dt\,.
\end{multline*}
 In this case, the nodal and global errors have the same rate of  convergence which is $O(k)$, see \cite[Theorem 3.8]{BrunnerSchoetzau06}.
 \end{remark}

For our error analysis, it will be convenient to reformulate the DG scheme (\ref{eq: DG step}) in terms of the global bilinear form
\begin{equation}\label{eq: GN def}
\begin{split}
G_N(U,X)&=U^0_+\,X^0_++\sum_{n=1}^{N-1}[U]^n\,X^n_+
    +\sum_{n=1}^N\int_{t_{n-1}}^{t_n} \Bigl[U'(t)+a(t)U(t)+\B U(t)\Bigr]X\,dt.
\end{split}
\end{equation}
By summing up (\ref{eq: DG step}) over all the time-steps and using $U^0_-=u_0$, the DG method can now equivalently be written as: Find $U\in
\W_p$ such that
\begin{equation}\label{eq:DGFEM}
G_N(U,X)=u_0\,X^0_++\int_0^{t_N}f\,X\,dt\quad\forall~~ X\in\W_p.
\end{equation}
Since the solution~$u$ is continuous, it follows that
\[
G_N(u,X)=u_0\,X^0_++\int_0^{t_N}f\,X\,dt\quad\forall~~ X\in\W_p.
\]
Thus, the following Galerkin orthogonality property holds:
\begin{equation}
\label{eq:Gal-orthog} G_N(U-u,X)=0\quad\text{$ \forall~ X\in\W_p$\,.}
\end{equation}
Before stating the regularity property of the solution $u$ of (\ref{eq: VIE}), we display  in the next remark an alternative form of $G_N$ which
will be used in our error analysis.
\begin{remark}
\label{rem:GN-alt} Integration by parts yields the following alternative expression for the bilinear form $G_N$ in~\eqref{eq: GN def}:
\begin{multline*}
G_N(U,X)=U_-^N\,X_-^N-\sum_{n=1}^{N-1}U_-^n\,[X]^n\\
+\sum_{n=1}^N\int_{t_{n-1}}^{t_n}\,\left[-U(t)X'
    +a(t)U(t)X+\B U(t)X\right]\,dt.
\end{multline*}
\end{remark}
Throughout the paper, we assume that the solution $u$ of~(\ref{eq: VIE}) satisfies:
\begin{equation} \label{eq:countable-regularity v1}
|u^{(j)}(t)|\le C\, t^{\sigma-j}\quad{\rm for}\quad 1\le j\le p+1~~{\rm where}~~1\le \sigma \le \alpha+1\,
\end{equation}
where the constant $C$ depends on $j$. For instance, if in (\ref{eq: VIE}) the function $f=t^{\kappa_1}f_1+ t^{\kappa_2} f_2$ for some
$\kappa_1,\,\kappa_2 \ge 0$ and the functions $a,\,b,\,f_1$ and $f_2$ are in ${\rm C}^{j-1}[0,T]$ for $1\le j\le p$, then
(\ref{eq:countable-regularity v1}) holds for $\sigma=1+\min\{\kappa_1,\kappa_2,\alpha\},$ see \cite[Section 7.1]{Brunner04} for more details.

We notice from (\ref{eq:countable-regularity v1}) that $|u^{(j)}(t)|$ is not bounded near $t=0$ for $j\ge 2$. Hence,  to compensate the singular
behavior of $u$ near $t=0$, we employ a family of non-uniform meshes, where the time-steps are concentrated near zero. Thus, we assume that, for
a fixed $\gamma\ge1$,
\begin{equation}\label{eq: kn gamma}
k_n\le C_\gamma kt_n^{1-1/\gamma} \quad\text{and}\quad t_n\le C_\gamma t_{n-1}\quad\text{for $2\le n\le N$,}
\end{equation}
with
\begin{equation}\label{eq: k1 gamma}
c_\gamma k^\gamma\le k_1\le C_\gamma k^\gamma.\end{equation} For instance, one may choose
\begin{equation}\label{eq: standard tn}
t_n=(n/N)^\gamma T\quad\text{for $0\le n\le N$.}
\end{equation}
Under the assumptions (\ref{eq:countable-regularity v1})--(\ref{eq: k1 gamma}), we show in Theorem~\ref{thm: U^n-u(t_n)} that the error
$|U_-^n-u(t_n)|$ is of order $k^{\gamma\sigma+\min\{p,1+\alpha\}}$, for $1\le n\le N$. So, we have a superconvergence of order
$k^{p+1+\min\{p,1+\alpha\}}$ provided $\gamma>(p+1)/\sigma$. However, for a quasi-uniform mesh (i.e., $\gamma=1$) our bound yields a poorer
convergence rate of order~$k^{\sigma+\min\{p,1+\alpha\}}$.
\section{Projection operator and technical lemmas}
\label{sec:time-error} In this section we introduce a projection operator that has been used various times in the analysis of DG time-stepping
methods; see~\cite{Thomee2006}, and state some preliminary results that are needed in our convergence analysis in the forthcoming sections.

For a given function $ u\in C[0,T]$, we define the interpolant ${\Pi}^- u \in \W_p$ by
\begin{equation}
\label{eq:Pi} {\Pi}^- u(t_n^-)=u(t_n)~~\text{and}~~ \int_{t_{n-1}}^{t_n}\, ( u\,-{\Pi}^- u)\,v \,dt=0\quad \text{ $\forall~~v \in
\Poly_{p-1}(I_n)$}
\end{equation}
 and for $1\le n\le N$. From~\cite[Lemma~3.2]{SchoetzauSchwab00} it follows that ${\Pi}^-$
is well-defined.

 To state the approximation properties of ${\Pi}^-$, we  introduce the
 notation
\[
\|\phi\|_{I_n}=\sup_{t\in I_n}|\phi(t)|\quad{\rm for~any}~~\phi \in C(t_{n-1},t_n).
\]
\begin{theorem}\label{thm:eta-approx}
 There exists a
constant $C$, which depends on $p$ such that:
\begin{romenum}
\item For any $0\le q\le p$ and $u|_{I_n}\in H^{q+1}(I_n)$, there
holds
\[
\int_{t_{n-1}}^{t_n}\, |{\Pi}^- u-u|^2\,dt\le Ck_n^{2q+2} \int_{t_{n-1}}^{t_n}\,|u^{(q+1)}|^2\,dt\quad {\rm for}~~ 1\leq n\leq N.
\]
\item For any $0\le q\le p$ and $u|_{I_n}\in H^{q+1}(I_n)\cap C(I_n)$, there
holds
\[
\|{\Pi}^- u-u\|^2_{I_n}\le C k_n^{2q+1} \int_{t_{n-1}}^{t_n}\,|u^{(q+1)}|^2\,dt\quad {\rm for}~~ 1\leq n\leq N.
\]
\end{romenum}
\end{theorem}
\begin{proof}
For the proof of the first bound, we refer to ~\cite[Section~3]{SchoetzauSchwab00} or \cite[Chapter 12, Page 214]{Thomee2006}. For the second
bound, see ~\cite[Theorem 3.9 and Corollary 3.10]{SchoetzauSchwabDGODE} or \cite[Equation (12.10)]{Thomee2006}\,.
\end{proof}

The following two technical lemmas are needed in our derivation of the  error estimates. The first lemma has been proved in~\cite[Lemma
6.3]{LarssonThomeeWahlbin98}.
\begin{lemma}  \label{lemma: 6.3LTW} If $g \in L_2(0,T)$ and
$\alpha \in (0,1)$ then
\[\int_0^T\left(\int_0^t(t-s)^{\alpha-1}
g(s)\,ds\right)^2dt \le \frac{T^\alpha}{\alpha}\int_0^T(T-t)^{\alpha-1} \int_0^t g^2(s)\,ds\,dt.\]
\end{lemma}
The next lemma is the following Gronwall inequality; see \cite[Lemma 6.4]{LarssonThomeeWahlbin98}.
\begin{lemma}\label{lemma: 6.4LTW} Let $\{a_j\}_{j=1}^N$ and
$\{b_j\}_{j=1}^N$ be sequences of non-negative numbers with $0\le b_1\le b_2\le \cdots\le b_N.$ Assume that there exists a constant $K\ge 0$
such that
\[a_n \le b_n+K\sum_{j=1}^n a_j \int_{t_{j-1}}^{t_j}(t_n-t)^{\alpha-1}\,dt
\quad{\rm for}~~1\leq n\leq N~~{\rm and}~~ \alpha\in (0,1). \] Assume further that $\delta=\frac{K\,k^\alpha}{\alpha}<1.$ Then for
$n=1,\cdots,N,$ we have $a_n \le C b_n$ where $ C$ is a constant that solely depends on $K$, $T$, $\alpha$ and $\delta$.
\end{lemma}

 Throughout the rest of the paper, we shall
always implicitly assume that the {\em maximum step-size $k$ is sufficiently small} so that the condition $\delta <1$ in Lemma~\ref{lemma:
6.4LTW} is satisfied. More precisely, following Lemma  \ref{lemma: bound1 of thetan}, we shall require that
\begin{equation*}
4\,T^\alpha\left(\frac{\mu^*}{\alpha\,\mu_*}\right)^2 k^\alpha <1\,.
\end{equation*}
\section{Error analysis of the dual problem}\label{sec:dual}
This section is devoted to deriving error estimates for the DG method applied to the dual problem of the Volterra integro-differential equation
(\ref{eq: VIE}). The main results of this section (more precisely, Theorem \ref{v2thm: U-Pi u}) play a crucial role in the proof of the
superconvergence
 error estimate in section \ref{sec:nodal-error}.

 Let $z$ be the solution of the dual problem
\begin{equation}\label{eq: homog dual}
-z'+a(t)z(t)+\B^\ast z(t)=0\quad\text{for $0\le t< T$,}\quad
    \text{with $z(T)=z_T$,}
\end{equation}
where $\B^\ast v(t)=\int_t^T\beta(s,t)v(s)\,ds$ ($\B^\ast$ is the dual of the integral operator $\B$)\,.

Since $z$ has no jumps and since
\begin{multline*}
\int_0^T\bigl[-v(t)z'(t)+a(t)v(t)z(t)+\B v(t)\,z(t)\bigr]\,dt
    \\=\int_0^Tv(t)(-z'(t)+a(t)z(t)+\B^\ast z(t))\,dt=0,
\end{multline*}
the alternative expression of $G_N$ given in Remark \ref{rem:GN-alt}
 yields the identity
\begin{equation}\label{v2eq: GN dual}
G_N(v,z)=v_-^Nz_T\quad{\rm forall}~~ v\in C[0,T]\,.
\end{equation}
($C(0,T]$ denotes the space of continuous functions on $[0,T]$). Let $Z\in \W_p$ denote the approximate solution of~\eqref{eq: homog dual} given
by
\begin{equation}\label{eq: DG solution backward}
G_N(V,Z)=V_-^N z_T\quad\text{ $\forall~~V\in \W_p$\,.}
\end{equation}
Hence, the following Galerkin orthogonality property holds:
\begin{equation}
\label{eq:Gal-orthog dual}
 G_N(V,Z-z)=0\quad\text{
$\forall~~V\in\W_p$}\,.
\end{equation}
At this stage, the main aim is to estimate the error $Z-z$ in the $L_2$-norm. First it is good to notice that  \eqref{eq:Gal-orthog dual} is a
discrete backward analogue of \eqref{eq:Gal-orthog}. Since it is more convenient to deal with a discrete forward problem, we introduce the
functions $\tilde z(t)=z(t_N-t)$ and $\tilde Z(t)=Z(t_N-t)$ and then,  \eqref{eq:Gal-orthog dual} can be rewritten as;
\begin{equation}
\label{eq:Gal-orthog dual forward}
 \tilde G_N( \tilde Z-\tilde z,V)=0\quad\text{
$\forall~~V\in\widetilde \W_p$;}
\end{equation}
 where  $\tilde G_N$ is defined  as in \eqref{eq: GN def}
 but with $\tilde a(t):=a(t_N-t)$ in place of $a(t)$ and $\beta(t_N-s,t_N-t)$ in place of $\beta(t,s)$. The finite
 dimensional space $\widetilde \W_p$ is defined as $\W_p$ but on the reverse
 mesh: $0=\tilde t_0<\tilde t_1<\cdots<\tilde t_N$, where $\tilde
 t_i=\tilde t_{i-1}+\tilde k_i$ with $\tilde k_i=k_{N+1-i}.$

 Setting $\zeta={\tilde \Pi}^- \tilde z-\tilde z$ and $\theta=\tilde Z-{\tilde \Pi}^- \tilde
z$ where $\tilde \Pi^-$ is the interpolant operator defined as in (\ref{eq:Pi}), but on the reverse mesh. Then (\ref{eq:Gal-orthog dual
forward}) implies that
\begin{equation}\label{eq: ortho v1}
\tilde G_N(\theta,V)=-\tilde G_N(\zeta,V)
    \quad\text{ $\forall~~
    V\in\W_p$\,.}
\end{equation}
By the construction of the interpolant we have $\zeta(\tilde t^n_-)=0$ for all $n\ge1$ and hence, using the alternative expression for~$G_N$
 given in Remark \ref{rem:GN-alt} and $\int_{\tilde t_{n-1}}^{\tilde t_n}\zeta(t)\,V'(t)\,dt=0$ (by definition of the
operator $\Pi^-$),
\begin{equation}\label{eq: orthogonal of the dual}
\tilde G_N(\zeta,V)=\sum_{n=1}^N\int_{\tilde t_{n-1}}^{\tilde t_n}\bigl[\tilde a(t)\zeta(t) V(t) +\tilde \B \zeta(t) V(t)\bigr]\,dt
\end{equation}
where \[\tilde \B \zeta(t)=\int_0^t \beta(t_N-s,t_N-t) \zeta(s)\,ds\,. \] In the next theorem we estimate the error between $z$ and $Z$.
\begin{theorem}\label{v2thm: U-Pi u} If $z$ is the solution of the
backward VIE~\eqref{eq: homog dual}, and if $Z \in \W_p$ is the approximate solution defined by ~\eqref{eq: DG solution backward}, then
\[ \int_0^{t_N}|z-Z|^2\,dt\le
     C k^{2\alpha+2}|z_T|^2\,\]
     provided that
     \begin{equation}\label{eq: bound1 of thetan}
\int_0^{t_N} |\theta(t)|^2\,dt
 \leq
 C\int_0^{t_N}|\zeta(t)|^2\,dt\,.
\end{equation}
\end{theorem}
\begin{proof}
From     the decomposition: $\tilde Z-\tilde z=\zeta+\theta$, the triangle inequality, and \eqref{eq: bound1 of thetan}, we have
\begin{equation}  \int_0^{t_N}|z-Z|^2\,dt= \int_0^{t_N}|\tilde z-\tilde Z|^2\,dt \le  C   \int_0^{t_N}|\zeta|^2\,dt.
\label{eq: estimate of dual error}
\end{equation}
Thus, the task reduces to bound the right-hand side of \eqref{eq: estimate of dual error}. Starting from the relation  $\tilde z(t)=z(t_N-t)$
and recalling that $z$ satisfies \eqref{eq: homog dual}, it is clear that $\tilde z$ solves the VIE:
 \[\tilde z'+a(t_N-t)\tilde z(t)+\int_0^t\beta(t_N-s,t_N-t)\tilde z(s)\,ds  = 0~~\text{
for~$0<t<T$,}
\]
with~$\tilde z(0)=z_T$. Hence, an application of (\ref{eq:countable-regularity v1}) for $\sigma=\alpha+1$ with $\tilde z$ in place of $u$ gives
\begin{equation}\label{v2lem: regularity of z}
\begin{aligned}
|\tilde z'(t)|+t^{1-\alpha}|\tilde z''(t)|+t^{2-\alpha}|\tilde z'''(t)|&\le C|z_T|\,.
\end{aligned}
\end{equation}
Now, Theorem \ref{thm:eta-approx} on the reverse mesh (with $\zeta$ in place of ${\Pi}^- u-u$) and \eqref{v2lem: regularity of z} yield
\begin{equation}\label{v2eq: bound of A-1e1}
\begin{aligned}
\sum_{n=2}^N\int_{\tilde t_{n-1}}^{\tilde t_n}|\zeta(t)|^2\,dt &\le C\sum_{n=2}^N \tilde k_n^4\int_{\tilde t_{n-1}}^{\tilde t_n}|\tilde
z''(t)|^2\,dt
\le C\sum_{n=2}^N \tilde k_n^4\int_{\tilde t_{n-1}}^{\tilde t_n}t^{2\alpha-2}|z_T|^2\,dt \\
&\le C|z_T|^2\sum_{n=2}^N \tilde k_n^5\tilde t_{n-1}^{2\alpha-2}= C|z_T|^2\sum_{n=2}^N \tilde k_n^{3+2\alpha}(\tilde k_n/\tilde t_{n-1})^{2-2\alpha}\\
&\le  C|z_T|^2\sum_{n=2}^N \tilde k_n^{3+2\alpha}\le
    Ck^{2\alpha+2}|z_T|^2\,
\end{aligned}
\end{equation}
 and on $(0,\tilde t_1)$, we notice for $1/2 < \alpha \le 1$ that
\begin{equation*}
\int_0^{\tilde t_1}|\zeta(t)|^2\,dt \le C \tilde k_1^4\int_0^{\tilde t_1}|\tilde z''(t)|^2\,dt \le C k_N^4\int_0^{\tilde
t_1}t^{2\alpha-2}|z_T|^2\,dt \le C|z_T|^2 k_N^{3+2\alpha}\,,
\end{equation*}
and  for $0<\alpha\le 1/2$ that
\begin{equation}\label{v2eq: bound of A-1e1 neq1}
\begin{aligned}
\int_0^{\tilde t_1}|\zeta(t)|^2\,dt \le C \tilde k_1^2\int_0^{\tilde t_1}|\tilde z'(t)|^2\,dt &\le C
k_N^2\int_0^{\tilde t_1}|z_T|^2\,dt \\
&\le C|z_T|^2 k_N^{3} \le C|z_T|^2 k_N^{2+2\alpha} \,.
\end{aligned}
\end{equation}
Finally, combine \eqref{eq: estimate of dual error} and  (\ref{v2eq: bound of A-1e1})--\eqref{v2eq: bound of A-1e1 neq1},  we obtain the desired
result.
\end{proof}
In the next lemma we prove the applicability of the assumption \eqref{eq: bound1 of thetan}.
\begin{lemma}\label{lemma: bound1 of thetan}
For $1\leq n\leq N$, we have
\[
\int_0^{\tilde t_n}|\theta(t)|^2\,dt
 \leq
 C\int_0^{\tilde t_n}|\zeta(t)|^2\,dt\,
\]
\end{lemma}
\begin{proof} Choosing $V=\theta$ on $(0,\tilde t_n)$ and zero elsewhere  in~ \eqref{eq: ortho v1} and \eqref{eq: orthogonal of the dual}, then using the
alternative definition of $G_N$ in Remark~\ref{rem:GN-alt} and $\theta'\theta=(d/dt)|\theta|^2/2$, we observe that
\begin{multline*}
|\theta(\tilde t^n_-)|^2+|{\theta}(\tilde t_0^+)|^2+\sum_{j=1}^{n-1}|[\theta]^j|^2+2\int_0^{\tilde t_n}\,\tilde a(t)|\theta(t)|^2\,dt
\\=-2\int_0^{\tilde t_n}\,\Bigl[\tilde a(t) \zeta(t)+\tilde B \zeta(t)
+\tilde \B \theta(t) \Bigr]\theta(t)\,dt.
\end{multline*} So
\[
 \int_0^{\tilde t_n}\,\tilde a(t)|\theta(t)|^2\,dt \leq
 \int_0^{\tilde t_n}\tilde a(t) |\theta(t)||\zeta(t)|\,dt+\int_0^{\tilde t_n}|\tilde \B
\zeta(t)+ \tilde \B \theta(t)|\,|\theta(t)|dt.
\]
We use the geometric-arithmetic mean inequality $|xy|\leq \frac{\varepsilon x^2}{2}+\frac{y^2}{2\varepsilon}$  (valid for any $\varepsilon>0$)
we find that
\begin{multline*}
 \int_0^{\tilde t_n}\tilde a(t)|\theta(t)||\zeta(t)| \,dt\le \sqrt{\mu^*}\int_0^{\tilde t_n}\sqrt{\tilde a(t)}|\theta(t)||\zeta(t)| \,dt\\ \le
\frac{1}{4}\int_0^{\tilde t_n}\tilde a(t)|\theta(t)|^2\,dt+\mu^*\int_0^{\tilde t_n}|\zeta(t)|^2\,dt \end{multline*} and thus
\begin{multline}\label{eq:q1+q2+q3}
 \frac{3}{4}\int_0^{\tilde t_n}\,\tilde a(t)|\theta(t)|^2\,dt \leq
 \mu^*\int_0^{\tilde t_n}|\zeta(t)|^2\,dt+\int_0^{\tilde t_n}|\tilde \B
\zeta(t)+ \tilde \B \theta(t)|\,|\theta(t)|dt.
\end{multline}

 We employ  the Cauchy-Schwarz inequality, again the geometric-arithmetic mean inequality, and Lemma \ref{lemma: 6.3LTW} (with
$T=\tilde t_n$):
\begin{align*}
\int_0^{\tilde t_n}|\B &\zeta(t)\,\theta(t)|dt\le \mu^* \int_0^{\tilde t_n}\int_0^t
(t-s)^{\alpha-1}|\zeta(s)|\,|\theta(t)|\,ds\,dt\\
&\le \frac{\mu^*}{\mu_*}\int_0^{\tilde t_n}\tilde a(t)^{1/2}|\theta(t)|\int_0^t (t-s)^{\alpha-1}\tilde a(s)^{1/2}|\zeta(s)|\,ds\,dt
\\
&\le  \frac{\mu^*}{\mu_*}\left(\int_0^{\tilde t_n}\left(\int_0^t (t-s)^{\alpha-1}\tilde a(s)^{1/2}|\zeta(s)|\,ds\right)^2\,dt\right)^{1/2}
\left(\int_0^{\tilde t_n}\tilde a(t)|\theta(t)|^2\,dt\right)^{1/2}\\
&\le \left(\frac{\mu^*}{\mu_*}\right)^2\int_0^{\tilde t_n}\left(\int_0^t (t-s)^{\alpha-1}\tilde a(s)^{1/2}|\zeta(s)|\,ds\right)^2\,dt+
\frac{1}{4}\int_0^{\tilde t_n}\tilde a(t)|\theta(t)|^2\,dt\\
&\le \frac{\tilde t_n^\alpha}{\alpha}\left(\frac{\mu^*}{\mu_*}\right)^2 \int_0^{\tilde t_n}(\tilde t_n-t)^{\alpha-1}\int_0^t\, \tilde
a(s)|\zeta(s)|^2\,ds\,dt+\frac{1}{4}\int_0^{\tilde t_n}\tilde
a(t)|\theta(t)|^2\,dt\\
&\le \left(\frac{{\tilde t_n}^{\alpha}\mu^*}{\alpha\,\mu_*}\right)^2 \int_0^{\tilde t_n} \tilde a(s)|\zeta(s)|^2\,ds+\frac{1}{4}\int_0^{\tilde
t_n}\tilde a(t)|\theta(t)|^2\,dt\,.
\end{align*}
Similarly, we notice that
\begin{multline*}
\int_0^{\tilde t_n}|\B \theta(t)\,\theta(t)|dt\\
\le\frac{\tilde t_n^\alpha}{\alpha}\left(\frac{\mu^*}{\mu_*}\right)^2 \int_0^{\tilde t_n}(\tilde t_n-t)^{\alpha-1}\int_0^t \tilde
a(s)|\theta(s)|^2\,ds\,dt+\frac{1}{4}\int_0^{\tilde t_n}\tilde a(t)|\theta(t)|^2\,dt.
\end{multline*}
Inserting the above bounds in (\ref{eq:q1+q2+q3}) implies that
\begin{multline*}
\int_0^{\tilde t_n}\tilde a(t)|\theta(t)|^2\,dt \\
\le  C\int_0^{\tilde t_n}|\zeta(t)|^2\,dt
 +4\frac{\tilde t_n^\alpha}{\alpha}\left(\frac{\mu^*}{\mu_*}\right)^2\sum_{j=1}^n
 \int_{\tilde t_{j-1}}^{\tilde t_j}(\tilde t_n-t)^{\alpha-1}\,dt\int_0^{\tilde t_j}
\tilde a(t)|\theta(t)|^2\,dt.
\end{multline*}
Therefore, the desired result now immediately follows after applying of the Gronwall inequality in Lemma \ref{lemma: 6.4LTW} and using the
assumption (\ref{eq: mu*}) on the function $\tilde a$ (instead of $a$)\,.
\end{proof}
\section{Superconvergence results}\label{sec:nodal-error}
 In this section, we study the
nodal error analysis of the DG solution $U$ defined by ({\ref{eq: DG step}) with  $U_-^0=u_0$. We derive error estimate of the DG solution,
giving rise to superconvergence  algebraic rates. Our analysis partially relies on the techniques introduced in~\cite[Chapter 12]{Thomee2006}
for parabolic problems.
\begin{theorem}\label{thm: U^n-u(t_n)}
Let $\alpha \in (0,1)$ in (\ref{eq: VIE1}). Let the solution $u$ of problem~\eqref{eq: VIE}  satisfy the regularity
property~\eqref{eq:countable-regularity v1} and let $U\in \W_p$ be the DG approximate solution defined by ~\eqref{eq: DG step} with $p\ge 1$. In
addition to the mesh assumption \eqref{eq: kn gamma} and \eqref{eq: k1 gamma}, we assume that $k_n\ge k_{n-1}$ for $1\le n\le N.$ Then
\begin{itemize} \item for $p=1$,
\[ \underset{1\le n\le
N}{\max}|U^n_--u(t_n)|\le Ck\times
\begin{cases}
k^{\gamma\sigma},&1\le\gamma\le 2/\sigma\\
k^2,&\gamma\ge 2/\sigma\end{cases}
\]
\item and for $p\ge 2,$ we have  \[ \underset{1\le n\le
N}{\max}|U^n_--u(t_n)|\le C\max\{1,{\rm log}\, n\}k^{\alpha+1}\times
\begin{cases}
k^{\gamma\sigma},&1\le\gamma\le (p+1)/\sigma\\
k^{p+1},&\gamma\ge (p+1)/\sigma\,.\end{cases}
\]\end{itemize}
\end{theorem}
\begin{proof}
From \eqref{eq: DG solution backward}, (\ref{v2eq: GN dual}), (\ref{eq:Gal-orthog}) and (\ref{eq:Gal-orthog dual}) (recall that $\eta={\Pi}^-
u-u$), we observe that
\begin{equation}\label{v2eq: duality trick}
\begin{aligned}
(U_-^N-u(t_N))z_T&=G_N(U,Z)-G_N(u,z)\\
&=G_N(u,Z-z)=G_N(\eta,z-Z).
    \end{aligned}
\end{equation}
The alternative expression for~$G_N$
 given in Remark \ref{rem:GN-alt} and the equality $\eta(t^n_-)=0$  show that
\begin{equation}\label{v2eq: duality trick conseq}
G_N(\eta,z-Z)=\delta_{1N}+\delta_{2N},
\end{equation}
where
\[\delta_{1N}=-\sum_{j=1}^N\int_{t_{n-1}}^{t_n}\eta\,(z-Z)'\,dt\quad{\rm and}\quad \delta_{2N}=\int_0^{t_N}(a(t)\eta(t)+\B
\eta(t))\, (z-Z)(t)\,dt.\] To bound $\delta_{1N}$ and $\delta_{2N}$, we start from the  regularity property (\ref{v2lem: regularity of z}) and
the relation $z(t)=\tilde z(t_N-t)$, and get
\begin{equation}\label{eq: reg of z}
\begin{aligned}
|z'(t)|+(t_N-t)^{1-\alpha}|z''(t)|+(t_N-t)^{2-\alpha}|z'''(t)|&\le C|z_T|\,.
\end{aligned}
\end{equation}
For $p=1,$ the orthogonality property of ${\Pi}^-$ yields
\begin{multline*}
-\delta_{1N}=\sum_{j=1}^N\int_{t_{n-1}}^{t_n}\eta(t)\,z'(t)\,dt =\sum_{j=1}^N\int_{t_{n-1}}^{t_n}\eta(t)\,[z'(t)-z'(t_n)]\,dt
\\=\sum_{j=1}^N\int_{t_{n-1}}^{t_n}
\eta(t)\,\int_t^{t_n} z''(s)\,ds\,dt
\end{multline*}
and hence, with the help of (\ref{eq: reg of z}) we have
\begin{equation}\label{v2eq: bound of A-1e1' p=1}
\begin{aligned}
|\delta_{1N}|&\le
C\|\eta\|_{J_N}\sum_{n=1}^{N}k_n\int_{t_{n-1}}^{t_n}|z''(t)|\,dt \\
    &\le Ck\|\eta\|_{J_N}\int_0^{t_N}(t_N-t)^{\alpha-1}|z_T|\,dt = Ck\|\eta\|_{J_N} |z_T|\,t_N^\alpha/\alpha\,.
\end{aligned}
\end{equation}
For $p\ge 2,$ again the orthogonality property of ${\Pi}^-$ gives
\begin{multline*}
\delta_{1N}=-\sum_{j=1}^N\int_{t_{n-1}}^{t_n}\eta(t)\,z'(t)\,dt =\sum_{j=1}^{N-1}\int_{t_{n-1}}^{t_n}\eta(t)\,[\Pi^+ z'(t)- z'(t)]\,dt
\\+\int_{t_{N-1}}^{t_N}\eta(t)\,\int_t^{t_n}z''(s)\,ds\,dt
\end{multline*} where $\Pi^+ z'$ is the
 discontinuous, piecewise-linear interpolant of $z'$ defined by
\[
\Pi^+ z'(t):=z'(t_{n-1})+\frac{\bar z^n-z'(t_{n-1})}{k_n/2}(t-t_{n-1})
    \quad\text{for $t\in I_n$\,}
\]
with  $\bar z^n:=k_n^{-1}\int_{I_n}z'(t)\,dt$ denote the mean value of~$z'$ over the subinterval~$I_n$.

   Elementary calculations show
that, for $t\in I_n$, the interpolation error has the representation
\begin{equation*}
\begin{aligned}
\Pi^+ z'(t)-z'(t)&=\int^t_{t_{n-1}}(s-t)z'''(s)\,ds+\frac{t-t_{n-1}}{k_n^2}
        \int_{I_n}(t_n-s)^2z'''(s)\,ds,
\end{aligned}
\end{equation*}
and so, by (\ref{eq: reg of z}),
\begin{equation*}
\begin{aligned}
|\delta_{1N}|&\le
C\|\eta\|_{J_N}\left(\sum_{n=1}^{N-1}k_n^2\int_{t_{n-1}}^{t_n}|z'''(t)|\,dt+Ck_N\int_{t_{N-1}}^{t_N}|z''(t)|\,dt\right) \\
        &\le C\|\eta\|_{J_N}|z_T|\left(\sum_{n=1}^{N-1}k_n^2\int_{t_{n-1}}^{t_n}(t_N-t)^{\alpha-2}\,dt +k_N\int_{t_{N-1}}^{t_N}(t_N-t)^{\alpha-1}\,dt \right)
   \\
 &\le     C\|\eta\|_{J_N}|z_T|(k^{\alpha+1}\log(t_N/k_N)+k_N^{\alpha+1})
\end{aligned}
\end{equation*}
where in the last step we used; $k_n\ge k_{n-1}$ for $n \ge 1$,  and
\begin{multline*}
\sum_{n=1}^{N-1}k_n^2\int_{t_{n-1}}^{t_n}(t_N-t)^{\alpha-2}\,dt
\\
\le C\sum_{n=1}^{N-1}k_n^{1+\alpha}\int_{t_{n-1}}^{t_n}(t_N-t)^{-1}\,dt \le Ck^{1+\alpha}\int_0^{t_{N-1}}(t_N-t)^{-1}\,dt.\end{multline*} To
bound $\delta_{2N}$, we use (\ref{eq: mu*}), integrating,  applying the Holder's inequality and then, using Theorem \ref{v2thm: U-Pi u}, we
notice that
\begin{align*}
|\delta_{2N}|&\le \mu^*\|\eta\|_{J_N}\int_0^{t_N}\left(|Z(t)-z(t)|+\int_0^t(t-s)^{\alpha-1}\,ds\,|Z(t)-z(t)|\right)\,dt
\\
&\le \mu^*\|\eta\|_{J_N}(1+t_N^\alpha/\alpha)\int_0^{t_N}|Z(t)-z(t)|\,dt
\\
&\le C\|\eta\|_{J_N}\left(\int_0^{t_N}|Z(t)-z(t)|^2\,dt\right)^{1/2}
 \le
 Ck^{\alpha+1}\|\eta\|_{J_N}|z_T|\,.
\end{align*}
Using Theorem \ref{thm:eta-approx}, the regularity assumption \eqref{eq:countable-regularity v1}, and the mesh assumption \eqref{eq: k1 gamma},
we get
\begin{equation*}
\|\eta\|^2_{I_1}\le Ck_1\int_0^{t_1}|u'(t)|^2\,dt\le Ck_1\int_0^{t_1}t^{2\sigma-2}\,dt
    =  C\frac{t_1^{2\sigma}}{2\sigma-1}\le C
    k^{2\gamma\sigma},
\end{equation*}
and for $n\ge2$, we use \eqref{eq: kn gamma} instead of \eqref{eq: k1 gamma}  and obtain
\begin{align*}
\|\eta\|_{I_n}^2&\le C k_n^{2p+1}
\int_{t_{n-1}}^{t_n}|u^{(p+1)}(t)|^2\,dt\\
&\le
 Ck_n^{2p+1}\int_{t_{n-1}}^{t_n} t^{2\sigma-2p-2}\,dt
    \le C\,k_n^{2p+2} t_n^{2\sigma-2p-2}
    \le C\,k^{2p+2}t_n^{2\sigma-(2p+2)/\gamma}.
\end{align*}
Finally, combine the above estimations from $\delta_{1N}$ and $\delta_{2N}$, and  recalling \eqref{v2eq: duality trick conseq} and \eqref{v2eq:
duality trick}  yield the desired bound for $n=N.$ For the nodal error at any time step $t_{n_0}$ with $1\le n_0\le N$, we follow the above
steps with $n_0$ in place of $N$, which will then complete the proof.
\end{proof}
\section{Super-convergence analysis for smooth kernels}\label{sec:smooth kernel}
In this section, we handle the nodal super-convergence error  analysis of the DG scheme \eqref{eq: DG step}  for problem (\ref{eq: VIE}) when
$\alpha\in   \N_0$ (so the kernel  is smooth). We use a uniform mesh with step-size $k=T/N$ where $k$ is assumed to be  sufficiently small. In
our analysis, we follow the derivations of Sections \ref{sec:dual} and \ref{sec:nodal-error} with some modifications. We assume that the
functions $a$, $b$ and $f$ are sufficiently regular such that the solution $u$ of (\ref{eq: VIE}) satisfies $|u^{(j)}(t)|\le C$ (and
consequently $|\tilde z^{(j)}(t)|\le C|z_T|$) for $1\le j\le p+1$ with $t \in (0,T]$. Thus, from Theorem \ref{thm:eta-approx} we notice that for
$n\ge1$,
\begin{align}\label{eq: bound of eta smooth}
\|\eta\|_{I_n}^2&\le C k^{2p+1} \int_{t_{n-1}}^{t_n}|u^{(p+1)}(t)|^2\,dt \le Ck^{2p+2}.
\end{align}

We start our analysis by deriving the error involved in approximating the solution $z$ of the backward VIE~\eqref{eq: homog dual}.
\begin{theorem}\label{v2thm: U-Pi u smooth} If $z$ is the solution of the
backward VIE~\eqref{eq: homog dual}, and if $Z \in \W_p$ is the approximate solution defined by ~\eqref{eq: DG solution backward}, then
\[ \int_0^{t_N}|z-Z|^2\,dt\le
     C k^{2p+2}|z_T|^2\,.\]
   \end{theorem}
\begin{proof} First, we recall \eqref{eq:q1+q2+q3} (over a uniform mesh)
\begin{equation}\label{eq:q1+q2+q3 smooth}
 \frac{3}{4}\int_0^{ t_n}\,\tilde a(t)|\theta(t)|^2\,dt \leq
 \mu^*\int_0^{ t_n}|\zeta(t)|^2\,dt+\int_0^{t_n}|\tilde \B
\zeta(t)+ \tilde \B \theta(t)|\,|\theta(t)|dt.
\end{equation}
Using the fact that $\alpha-1\ge 0$, and  the Cauchy-Schwarz and  the geometric-arithmetic mean inequalities, we observe
\begin{align*}
\int_0^{ t_n}|\tilde \B &\zeta(t)\,\theta(t)|dt\le \mu^* \int_0^{ t_n}t^{\alpha-1}|\theta(t)|\int_0^t
|\zeta(s)|\,ds\,dt\\
&\le \frac{\mu^*}{\mu_*}\int_0^{ t_n}t^{\alpha-1} \tilde a(t)^{1/2}|\theta(t)|\int_0^t \tilde a(s)^{1/2}|\zeta(s)|\,ds\,dt
\\
&\le \frac{\mu^*}{\mu_*}\int_0^{ t_n}t^{\alpha-\frac{1}{2}}\, \tilde a(t)^{1/2}|\theta(t)|\left(\int_0^t \tilde
a(s)|\zeta(s)|^2\,ds\right)^{1/2}\,dt
\\
&\le \left(\frac{\mu^*}{\mu_*}\right)^2 \int_0^{ t_n}t^{2\alpha-1}\int_0^t\, \tilde a(s)|\zeta(s)|^2\,ds\,dt+\frac{1}{4}\int_0^{ t_n}\tilde
a(t)|\theta(t)|^2\,dt\\
&\le \frac{ t_n^{2\alpha}}{2\alpha}\left(\frac{\mu^*}{\mu_*}\right)^2 \int_0^{ t_n} \tilde a(s)|\zeta(s)|^2\,ds+\frac{1}{4}\int_0^{ t_n}\tilde
a(t)|\theta(t)|^2\,dt\,.
\end{align*}
Similarly, we notice that
\begin{multline*}
\int_0^{ t_n}|\tilde \B \theta(t)\,\theta(t)|dt\\
\le \left(\frac{\mu^*}{\mu_*}\right)^2 \int_0^{ t_n}t^{2\alpha-1}\int_0^t \tilde a(s)|\theta(s)|^2\,ds\,dt+\frac{1}{4}\int_0^{ t_n}\tilde
a(t)|\theta(t)|^2\,dt.
\end{multline*}
Inserting the above bounds in (\ref{eq:q1+q2+q3 smooth}) yields
\begin{multline*}
\int_0^{ t_n}\tilde a(t)|\theta(t)|^2\,dt \\
\le  C\int_0^{ t_n}|\zeta(t)|^2\,dt
 + t_n^\alpha\left(\frac{\mu^*}{\mu_*}\right)^2\sum_{j=1}^n
 \int_{ t_{j-1}}^{ t_j}t^{\alpha-1}\,dt\int_0^{ t_j}
\tilde a(s)|\theta(s)|^2\,ds.
\end{multline*}
Since one can show by induction on $\alpha$ ($\alpha \in \N_0$) that
 \[\int_{ t_{j-1}}^{ t_j}t^{\alpha-1}\,dt=\frac{1}{\alpha}[
t_{j}^{\alpha}- t_{j-1}^{\alpha}]= \frac{k^\alpha}{\alpha}[j^\alpha-(j-1)^\alpha] \le  k^\alpha j^{\alpha-1}\le k,
\] an application  of the standard discrete Gronwall Lemma  gives
\[
\int_0^{t_n}|\theta(t)|^2\,dt
 \leq
 C\int_0^{ t_n}|\zeta(t)|^2\,dt\,\quad{\rm for}~~1\le n\le N.
\]
Hence, (\ref{eq: estimate of dual error}) is valid now and therefore,  we obtain the desired result after noting that
\begin{equation*}
\begin{aligned}
\sum_{n=1}^N\int_{ t_{n-1}}^{ t_n}|\zeta(t)|^2\,dt &\le C\sum_{n=1}^N  k^{2p+2}\int_{t_{n-1}}^{t_n}|\tilde z^{(p+1)}(t)|^2\,dt
\le
    Ck^{2p+2}|z_T|^2\,.
\end{aligned}
\end{equation*}
\end{proof}

In the next theorem we study the nodal error analysis of the DG solution $U$ defined by ({\ref{eq: DG step}) with  $U_-^0=u_0$.
\begin{theorem}\label{thm: U^n-u(t_n) smooth}
Let $\alpha \in \N_0$ in (\ref{eq: VIE1}). Let the solution $u$ of problem~\eqref{eq: VIE} be sufficiently regular  and  let  $U\in \W_p$ be the
DG approximate solution defined by ~\eqref{eq: DG step} with $p\ge 1$. Then we have
\[ \underset{1\le n\le
N}{\max}|U^n_--u(t_n)|\le Ck^{2p+1}\,.
\]
\end{theorem}
\begin{proof}
We follow  the steps given in the proof of Theorem \ref{thm: U^n-u(t_n)}, however we  use the new bounds of $\delta_{1N}$ and $\delta_{2N}$
derived below. The orthogonality property of ${\Pi}^-$ gives
\[
\delta_{1N}=-\sum_{j=1}^N\int_{t_{n-1}}^{t_n}\eta(t)\,z'(t)\,dt =\sum_{j=1}^{N}\int_{t_{n-1}}^{t_n}\eta(t)\,[\widehat \Pi z'(t)- z'(t)]\,dt
\] where $\widehat \Pi z' \in \W_{p-1}$ is defined by: for $1\le n\le N$,
\[ \widehat {\Pi} z'(t_n^-)=z'(t_n)~~\text{and}~~ \int_{t_{n-1}}^{t_n}\, ( z'\,-\widehat {\Pi} z')\,v \,dt=0\quad \text{ $\forall~v \in
\Poly_{p-2}(I_n)$ }.
\]
 Hence, from Theorem \ref{thm:eta-approx},
 there exists a
constant $C$, which depends on $p$ such that:
\[
\|\widehat {\Pi} z'-z'\|^2_{I_n}\le C k^{2p-1} \int_{t_{n-1}}^{t_n}\,|z^{(p+1)}|^2\,dt\le Ck^{2p}|z_T|^2\quad {\rm for}~~ 1\leq n\leq N
\]
and thus, using (\ref{eq: bound of eta smooth}) we obtain
\[
|\delta_{1N}|\le
Ck\sum_{n=1}^{N} \|\eta\|_{I_n}\|\widehat \Pi z'(t)- z'(t)\|_{I_n}\le  \\
         Ck^{2p+1}|z_T|\,.\]
         For the
bound of $\delta_{2N}$, we use (\ref{eq: mu*}), integrating,  applying the Cauchy-Schwarz inequality and then, using \eqref{eq: bound of eta
smooth} and  Theorem \ref{v2thm: U-Pi u smooth}, we notice that
\begin{align*}
|\delta_{2N}|&\le \mu^*\|\eta\|_{J_N}\int_0^{t_N}\left(|Z(t)-z(t)|+\int_0^t(t-s)^{\alpha-1}\,ds\,|Z(t)-z(t)|\right)\,dt
\\
&\le \mu^*\|\eta\|_{J_N}(1+t_N^\alpha/\alpha)\int_0^{t_N}|Z(t)-z(t)|\,dt
\\
&\le C\|\eta\|_{J_N}\left(\int_0^{t_N}|Z(t)-z(t)|^2\,dt\right)^{1/2}
 \le
 Ck^{2p+2}|z_T|\,.
\end{align*}
\end{proof}
\section{Numerical examples}\label{sec:numerics}
In this section, we present a set of numerical experiments to demonstrate the obtained theoretical error estimates  and also, to justify the
validation of the DG scheme (\ref{eq: DG step}) for a wider class of integro-differential equations.

Throughout, we consider problem ~\eqref{eq: VIE} with  $T=1$, the initial data $u_0=0$ and  $b(t)=1/\Gamma(\alpha)$ (Here, $\Gamma$ denotes the
usual gamma function.). Recall that, $u$ denotes the exact solution of (\ref{eq: VIE}) and $U$ is the DG solution defined by (\ref{eq: DG step})
using $2^i$ ($i \ge 1$) subintervals, that is $N=2^i$\,.
\subsection{Example 1}
 Choosing  the coefficient $a(t)$  and the source term~$f(t)$ such that the solution $u$ of (\ref{eq: VIE}) is given by
\begin{equation}\label{eq: num ex1}
u(t)=t^{\alpha+1}{\rm e}^{-t}.
\end{equation}
For $\alpha \in (0,1)$,  we notice that  near $t=0$, $u''$ is not bounded, however $u$ is smooth away from $t=0.$
 So, we
employ a time mesh of the form~\eqref{eq: standard tn} for various choices of the mesh grading parameter~$\gamma\ge1$ to verify the results of
Theorem \ref{thm: U^n-u(t_n)}.

Since the exact solution~\eqref{eq: num ex1} behaves like~$t^{\alpha+1}$ as $t\to0^+$, we see that the regularity
condition~\eqref{eq:countable-regularity v1} holds for ~$\sigma= \alpha+1$. Thus, from Theorem \ref{thm: U^n-u(t_n)} and by ignoring the
logarithmic factor, we expect
\begin{multline*}
\|U-u\|_{node}:=\max_{1\le n\le N}|U^N_--u(t_n)|\\
=\begin{cases}
 O(k^{\gamma(\alpha+1)+\min\{p,\alpha+1\}})~&{\rm for}~
1\le\gamma<(p+1)/(\alpha+1), \\
O(k^{p+1+\min\{p,\alpha+1\}})~&{\rm for}~\gamma \ge (p+1)/(\alpha+1).\end{cases}\end{multline*}
 \underline{\bf Case 1} Choosing  $a(t)=1$, thus
\begin{equation}\label{eq: data ex1} f(t)=(\alpha+1)t^\alpha {\rm e}^{-t}+t^{2\alpha+1}\sum_{i=0}^\infty (-1)^i \frac{t^i}{i!} \frac{
\Gamma(2+\alpha+i)}{\,\Gamma(2+2\alpha+i)}\,.\end{equation}

To illustrate the theoretical results of Theorem \ref{thm: U^n-u(t_n) smooth}, we choose $\alpha=2$ and so, the memory term and the solution $u$
are smooth. As expected, the numerical results in Table \ref{tab: ||U-u|| nodal 0} demonstrate nodal errors of order
 $O(k^{2p+1})$  for $p=1,2,3$.

 In Tables \ref{tab: ||U-u|| nodal
1}--\ref{tab: ||U-u|| nodal 3} we displayed the nodal error $\|U-u\|_{node}$ over the mesh (\ref{eq: standard tn}) with $N=2^i$ and for
different values of $\gamma$ when $\alpha=0.2$ and $\alpha=0.5$ (So $|u^{(j)}(t)|$ is not bounded near $t=0$ for $j\ge 2.$). Results shown in
these tables confirm that the best convergence rate we can achieve is $O(k^{p+1+\min\{p,\alpha+1\}})$ and thus our theoretical results in
Theorem \ref{thm: U^n-u(t_n)} are sharp in terms of the convergence order. However, it indicates that in practice we can relax the restriction
on the mesh grading exponent~$\gamma$. We conjecture that $\gamma \ge (p+1+\min\{p,\alpha+1\})/(\sigma+\alpha+1)$ suffices to ensure
$O(k^{p+1+\min\{p,\alpha+1\}})$ convergence. More precisely, we observe $O(k^{(\sigma+\alpha+1)\gamma})$-rates  if $1\le \gamma \le
(p+1+\min\{p,\alpha+1\})/(\sigma+\alpha+1)$.}
\begin{table}
\renewcommand{\arraystretch}{1}
\begin{center}
\begin{tabular}{|r|rr|rr|rr|}
\hline  $i$ &\multicolumn{2}{c|}{$p=1$}&\multicolumn{2}{c|}{$p=2$} &\multicolumn{2}{c|}{$p=3$}
\\
\hline
2&  3.953e-05&  2.622&  1.675e-07& 4.934& 1.928e-10& 6.949\\
3&  5.430e-06&  2.864&  5.391e-09& 4.958& 1.537e-12& 6.971\\
4&  7.063e-07&  2.943&  1.712e-10& 4.976& 1.199e-14& 7.002\\
5&  8.991e-08&  2.974&  5.409e-12& 4.985& 1.003e-16& 6.902\\
\hline\end{tabular}\vline\caption {The nodal error  and the convergence rate over a uniform mesh with $N=2^i$ subintervals when $\alpha=2$ in
(\ref{eq: num ex1})--(\ref{eq: data ex1}).} \label{tab: ||U-u|| nodal 0}
\end{center}
\end{table}
\begin{table}
\renewcommand{\arraystretch}{1}
\begin{center}
\begin{tabular}{|r|rr|rr|rr|}
\hline $i$ &\multicolumn{2}{c|}{$\gamma=1$}&\multicolumn{2}{c|}{$\gamma=1.25$} &\multicolumn{2}{c|}{$\gamma=1.4$}
\\
\hline
6& 6.839e-08& 1.886& 3.991e-08& 3.076&4.883e-08& 3.085\\
7& 1.522e-08& 2.168& 4.746e-09& 3.071&5.767e-09& 3.082\\
8& 3.118e-09& 2.286& 5.668e-10& 3.066&6.836e-10& 3.076\\
9& 6.147e-10& 2.342& 6.796e-11& 3.060&8.136e-11& 3.070\\
\hline
\end{tabular}
 \caption {The nodal error
 and the rate of convergence for Case 1  when $\alpha=0.2$ and $p=1$.} \label{tab: ||U-u|| nodal 1}
\end{center}
\end{table}
In Table \ref{tab: ||U-u|| nodal 1}, we have chosen $\alpha=0.2$ in (\ref{eq: num ex1})--(\ref{eq: data ex1}) and the DG solution $U\in \W_1$
(i.e., the approximate solution is a piecewise linear polynomial). An $O(k^{(\sigma+\alpha+1)\gamma})$ (i.e., $O(k^{2.4\gamma})$) convergence
rate has been observed if $1\le \gamma < 3/(\sigma+\alpha+1)$ and $O(k^3)$ if $\gamma \ge 3/(\sigma+\alpha+1)$. In Table \ref{tab: ||U-u|| nodal
2}, we considered $\alpha=0.5$ and  $U\in \W_p$ where $p=2$ or  $3$. An $O(k^{(\sigma+\alpha+1)\gamma})$ convergence rate has been demonstrated
if $1\le \gamma \le (p+2+\alpha)/(\sigma+\alpha+1)$. Finally, we chose $\gamma
> (p+2+\alpha)/(\sigma+\alpha+1)$ in Table \ref{tab: ||U-u|| nodal 3} and we realized that the order of
convergence almost matched the one given in the last column of Table \ref{tab: ||U-u|| nodal 2} where $\gamma = (p+2+\alpha)/(\sigma+\alpha+1)$
(i.e., the order of convergence did not exceed $p+2+\alpha$ for $p\ge 2$ as the theoretical results suggested).
\begin{table}
\renewcommand{\arraystretch}{1}
\begin{center}
\vline \vline
\begin{tabular}{|r|r|rr|rr|rr|}
 \hline & $i$ &\multicolumn{2}{c|}{$\gamma=1$}
&\multicolumn{2}{c|}{$\gamma=4/3$}
&\multicolumn{2}{c|}{$\gamma=(p+2.5)/3$}\\
\hline
     &6& 5.19e-10& 3.08& 8.23e-12& 4.09& 4.43e-12& 4.45\\
$p=2$&7& 6.28e-11& 3.04& 5.01e-13& 4.04& 2.00e-13& 4.46\\
     &8& 7.73e-12& 3.02& 3.10e-14& 4.01& 9.01e-15& 4.47\\
 \hline \hline
     &4& 2.36e-09& 3.13& 1.40e-10& 4.07& 1.80e-11& 5.47\\
$p=3$&5& 2.83e-10& 3.06& 8.60e-12& 4.03& 4.01e-13& 5.48\\
     &6& 3.47e-11& 3.03& 5.33e013& 4.01& 8.90e-15& 5.49\\
\hline
\end{tabular}\vline \vline
\end{center}
 \caption {The nodal error
 and the convergence rate for Case 1  when $\alpha=0.5$ and $p=2\,,3.$} \label{tab: ||U-u|| nodal 2}
\end{table}
\begin{table}
\renewcommand{\arraystretch}{1}
\begin{center}
\vline \vline
\begin{tabular}{|r|r|rr|}
  \hline  & $i$ &Error&  Rate\\
\hline
             &6& 4.6578e-12& 4.4511\\
$p=2$, $\gamma=5/3$        &7& 2.1039e-13& 4.4685\\
 &8& 9.3953e-15& 4.4850\\
 \hline \hline
           &3& 1.1677e-09&  5.4859\\
$p=3$, $\gamma=2$      &4& 2.5328e-11&  5.5269\\
 &5& 5.6022e-13&  5.4986\\
\hline
\end{tabular}\vline \vline
\end{center}
\caption {Nodal errors
 and  convergence rates for Case 1   when $\alpha=0.5$ and $p=2\,,3$.} \label{tab:
||U-u|| nodal 3}
\end{table}\\

\underline{\bf Case 2} Choosing $a(t)=t^\alpha+1$ and so
\begin{equation}\label{eq: data ex1 v1}
f(t)=(\alpha+1)t^\alpha {\rm e}^{-t}+t^{2\alpha+1} {\rm e}^{-t}+t^{2\alpha+1}\sum_{i=0}^\infty (-1)^i \frac{t^i}{i!} \frac{
\Gamma(2+\alpha+i)}{\,\Gamma(2+2\alpha+i)}\,.\end{equation}

 In Tables \ref{tab: ||U-u|| nodal 1 v1} and \ref{tab: ||U-u|| nodal 2 v1}  we displayed the nodal error $\|U-u\|_{node}$
 over the mesh (\ref{eq: standard tn}) with $N=2^i$ and for
different values of $\gamma$. Again, we observe convergence of order  $O(k^{(\sigma+\alpha+1)\gamma})$   if $1\le \gamma <
(p+1+\min\{p,\alpha+1\})/(\sigma+\alpha+1)$ and of order $O(k^{p+1+\min\{p,\alpha+1\}})$   if $\gamma \ge
(p+1+\min\{p,\alpha+1\})/(\sigma+\alpha+1)$ for different polynomial degrees $p$.
\begin{table}
\renewcommand{\arraystretch}{1}
\begin{center}
\begin{tabular}{|r|rr|rr|rr|}
\hline $i$ &\multicolumn{2}{c|}{$\gamma=1$}&\multicolumn{2}{c|}{$\gamma=1.25$} &\multicolumn{2}{c|}{$\gamma=1.4$}
\\
\hline
6& 1.633e-07& 2.305& 9.644e-08& 3.024& 1.233e-07& 3.024\\
7& 3.205e-08& 2.350& 1.184e-08& 3.026& 1.514e-08& 3.026\\
8& 6.220e-09& 2.365& 1.454e-09& 3.026& 1.858e-09& 3.026\\
9& 1.205e-09& 2.367& 1.787e-10& 3.024& 2.284e-10& 3.024\\
\hline
\end{tabular}
 \caption {The nodal error
 and the rate of convergence  for Case 2 when $\alpha=0.2$ and $p=1.$} \label{tab: ||U-u|| nodal 1 v1}
\end{center}
\end{table}
\begin{table}
\renewcommand{\arraystretch}{1}
\begin{center}
\vline \vline
\begin{tabular}{|r|r|rr|rr|rr|}
 \hline & $i$ &\multicolumn{2}{c|}{$\gamma=1$}
&\multicolumn{2}{c|}{$\gamma=4/3$}
&\multicolumn{2}{c|}{$\gamma=(p+2.5)/3$}\\
\hline
     &6& 1.55e-09& 3.01& 2.62e-11& 4.01& 6.94e-12& 4.40\\
$p=2$&7& 1.92e-10& 3.01& 1.63e-12& 4.01& 3.21e-13& 4.43\\
     &8& 2.39e-11& 3.01& 1.02e-13& 4.00& 1.47e-14& 4.45\\
 \hline \hline
     &4& 5.90e-09& 2.62& 3.76e-10& 3.68& 1.63e-11& 5.49\\
$p=3$&5& 8.22e-10& 2.84& 2.60e-11& 3.85& 3.59e-13& 5.50\\
     &6& 1.08e-10& 2.93& 1.71e-12& 3.93& 8.07e-15& 5.47\\
  \hline
\end{tabular}\vline \vline
\end{center}
 \caption {The nodal error
 and the convergence rate for Case 2   when $\alpha=0.5$ and $p=2\,,3.$.} \label{tab: ||U-u|| nodal 2 v1}
\end{table}

\subsection{Example 2}
In this example we demonstrate that the nodal superconvergence results  of Theorem \ref{thm: U^n-u(t_n)} are still valid even if $a(t) \equiv 0$
in (\ref{eq: VIE}) (so the assumption (\ref{eq: mu*}) is not satisfied) with $\alpha \in (0,1)$.

 In this case, (\ref{eq:
VIE}) reduces to the following (scalar evolution or fractional wave equation, see \cite{McLeanMustapha2007, MustaphaMcLean2009}) time-dependent
problem: for $\alpha\in(0,1)$,
\begin{equation}\label{eq: scalar evolution}
u'+\int_0^t \frac{(t-s)^{\alpha-1}}{\Gamma(\alpha)} u(s)\,ds=f(t)
    \quad \text{for $0<t<T$ with $u(0)=0$}\,.
\end{equation}
The piecewise linear ($p=1$) DG method for (\ref{eq: scalar evolution}) had been studied extensively in \cite{MustaphaMcLean2009}. However, for
$p\ge 2,$  the stability and convergence analyses of the DG method for (\ref{eq: scalar evolution}) are more difficult and it will be a topic of
future research.

 Using the Mittag--Leffler function
$E_\mu(x)=\sum_{p=0}^\infty x^p/\Gamma(1+p\mu)$, we may write the exact solution as
\[
u(t)=\int_0^t E_{\alpha+1}(-s^{\alpha+1})f(t-s)\,ds\,.
\]
 Choosing  a source term~$f(t)= (\alpha+1)t^{\alpha}$, we find
that
\begin{equation}\label{eq: num ex1 v1}
u(t)=-\Gamma(\alpha+2)\sum_{p=1}^\infty\frac{(-t)^{(\alpha+1)p}} {\Gamma(1+(\alpha+1)p)} =\Gamma(\alpha+2)
\left(1-E_{\alpha+1}(-t^{\alpha+1})\right).
\end{equation}

Since the exact solution of \eqref{eq: scalar evolution} behaves like~$t^{\alpha+1}$ as $t\to0^+$, we see that the  regularity
conditions~\eqref{eq:countable-regularity v1}  hold for any~$\sigma=\alpha+1$. For $p=1$ (that is, piecewise linear DG method),  the numerical
results shown in Table \ref{tab: nodal errors} demonstrate a nodal superconvergence  rate  of order $O(k^{\gamma(\sigma+\alpha+1)})$ for
$1\le\gamma<3/(\sigma+\alpha+1)$, and of order $O(k^3)$ for $\gamma \ge 3/(\sigma+\alpha+1)$. However, for $p\ge 2$, the numerical results shown
in Table~\ref{tab: ||U-u|| nodal 2 fractional wave} illustrated a nodal error estimates of  order $O(k^{\gamma(p+2+\alpha})$ (that is,
$O(k^{\gamma(\sigma+p+1})$) for $1\le\gamma<(p+2+\alpha)/(\sigma+\alpha+1)$, and almost of order $O(k^{p+2+\alpha})$  for $\gamma \ge
(p+2+\alpha)/(\sigma+\alpha+1)$.
\begin{table}
\renewcommand{\arraystretch}{1}
\begin{center}
\begin{tabular}{|r|rr|rr|rr|rr|}
\hline $i$&\multicolumn{2}{c|}{$\gamma=1$} &\multicolumn{2}{c|}{$\gamma=1.25$}
&\multicolumn{2}{c|}{$\gamma=1.5$} \\
\hline
  6& 9.11e-08& 2.33& 1.70e-08& 2.79& 2.59e-08& 2.76\\
  7& 1.76e-08& 2.37& 2.38e-09& 2.84& 3.67e-09& 2.82\\
  8& 3.37e-09& 2.39& 3.24e-10& 2.87& 5.05e-10& 2.86\\
  9& 6.40e-10& 2.39& 4.34e-11& 2.90& 6.83e-11& 2.89\\
\hline
\end{tabular}
\end{center}
\caption {The nodal error  and the convergence rates for Example 2, when $\alpha=0.2$ and $p=1.$} \label{tab: nodal errors}
\end{table}
\begin{table}
\renewcommand{\arraystretch}{1}
\begin{center}
\vline \vline
\begin{tabular}{|r|r|rr|rr|rr|}
 \hline & $i$ &\multicolumn{2}{c|}{$\gamma=1$}
&\multicolumn{2}{c|}{$\gamma=4/3$}
&\multicolumn{2}{c|}{$\gamma=(p+2.5)/3$}\\
\hline
      &5& 3.94e-09& 3.03& 1.27e-10& 4.32& 1.77e-10& 4.47\\
$p=2$ &6& 4.89e-10& 3.01& 7.89e-12& 4.01& 7.89e-12& 4.48\\
      &7& 6.09e-11& 3.00& 4.92e-13& 4.00& 3.51e-13& 4.49\\
 \hline \hline
      &4& 2.17e-09& 3.00& 1.36e-10& 4.00& 2.23e-11& 5.35\\
$p=3$ &5& 2.72e-10& 3.00& 8.49e-12& 4.00& 5.74e-13& 5.28\\
      &6& 3.40e-11& 3.00& 5.31e013& 4.00& 2.50e-14& 5.12\\
\hline
\end{tabular}\vline \vline
\end{center}
 \caption {The nodal error
 and the convergence rate for Example 2   when $\alpha=0.5$ and $p=2\,,3.$} \label{tab: ||U-u|| nodal 2 fractional wave}
\end{table}
\bibliographystyle{amsplain}

\begin{thebibliography}{99}

\bibitem{Brunner04} H. Brunner, Collocation Method for
Volterra Integral and Related Functional Differential Equations, Cambridge University Press, Cambridge, 2004.

\bibitem{BrunnerSchoetzau06}  H. Brunner  and D. Sch{\"{o}}tzau,
$hp$-Discontinuous Galerkin time stepping for Volterra integrodifferential equations, \textit{SIAM J. Numer. Anal.}, \textbf{44} (2006),
224--245.

\bibitem{BrunnerPedasVainikko01} H. Brunner, A. Pedas
and G. Vainikko, The piecewise polynomial collocation methods for linear Volterra integrodifferential equations with weakly singular kernels,
\textit{SIAM J. Numer. Anal.}, \textbf{39} (2001), 957--982.

\bibitem{DelfourHagerTrochu81} M. Delfour and W. Hager
and F. Trochu, Discontinuous Galerkin methods for ordinary differential equations, \textit{Math. Comp.}, \textbf{36} (1981), 455--473.

\bibitem{ErikssonJohnsonThomee1985} K. Eriksson and  C.
Johnson and Thom\'ee, Time discretization of parabolic problems by the discontinuous Galerkin method, \textit{RAIRO Mod\'el. Math. Anal.
Num\'er.}, \textbf{19} (1985), 611--643.

\bibitem{Estep95} D. Estep, A posteriori error bounds and global error
control for approximation of ordinary differential equations, \textit{SIAM J. Numer. Anal.}, \textbf{32} (1995), 1--48.


\bibitem{16} Y. J. Jiang, On spectral methods for Volterra-type Integro-differential equations,
\textit{J. Comput. Appl. Math.}, \textbf{230} (2009), 333--340.

\bibitem{Johnson88} C. Johnson, Error estimates and adaptive
time-step control for a class of one-step methods for stiff ordinary differential equations, \textit{SIAM J. Numer. Anal.}, \textbf{25} (1988),
908--926.


\bibitem{LarssonThomeeWahlbin98} S. Larsson,
V.~Thom\'ee and L. Wahlbin, Numerical solution of parabolic integro-differential equations by the discontinuous Galerkin method, \textit{Math.
Comp.}, \textbf{67} (1998), 45--71.

\bibitem{LesaintRaviart74} P. Lesaint and P.A. Raviart,
 On a finite element method for solving the neutron
transport equation in  {M}athematical {A}spects of {F}inite {E}lements in {P}artial {D}ifferential {E}quations ({M}adison, 1974), editor: C. de
Boor, Academic Press, New York, (1974), 89--145.

\bibitem{McLeanMustapha2007} W. Mclean and K. Mustapha, A second-order accurate
numerical method for a fractional wave equation, \textit{Numer. Math.}, \textbf{105} (2007), 481--510.

\bibitem{Mustapha2008} K. Mustapha, A Petrov-Galerkin method for integro-differential equations with a memory term,
\textit{ANZIAM J.}, \textbf{50} (2008), 610--624.

\bibitem{MustaphaMcLean2009}
K. Mustapha and W. McLean,  Discontinuous Galerkin method for an evolution equation with a memory term of positive type, \textit{ Math. Comp.},
\textbf{78} (2009), 1975--1995.

\bibitem{Mustapha2010}
K. Mustapha and H. Mustapha, A second-order accurate numerical method for a semilinear integro-differential equation with a weakly singular
kernel, \textit{IMA J. Numer. Anal.}, \textbf{30} (2010), 555--578.


\bibitem{MustaphaBrunnerMustaphaSchoetzau} K. Mustapha,
H. Brunner, H. Mustapha  and D. Sch{\"{o}}tzau, An hp-version discontinuous Galerkin method for integro-differential equations of parabolic
type, \textit{SIAM J. Numer. Anal.}, \textbf{49} (2011), 1369--1396.

\bibitem{Pani2010} A. Pani, G. Fairweather and R. Fernandes,
ADI orthogonal spline collocation methods for parabolic partial integro-differential equations, \textit{IMA J. Numer. Anal.}, \textbf{30}
(2010),  248--276.

\bibitem{Pani2011} A. Pani and S. Yadav, An $hp$-local discontinuous Galerkin method for parabolic integro-differential equations,
\textit{J. Sci. Comput.}, \textbf{46} (2011), 71–-99.

\bibitem{ReedHill73} W.H. Reed and T.R. Hill,
Triangular mesh methods for the neutron transport equation. {\rm Los Alamos Scientific Laboratory, LA-UR-73-479, 1973}.



\bibitem{SchoetzauSchwab00} D. Sch{\"{o}}tzau and C.
Schwab, Time discretization of parabolic problems by the $hp$-version of the discontinuous Galerkin finite element method, \textit{SIAM J.
Numer. Anal.}, \textbf{38} (2000), 837--875.

\bibitem{SchoetzauSchwabDGODE} D.~Sch{\"{o}}tzau and
C.~Schwab, An hp a-priori error analysis of the {DG} time-stepping method for initial value problems, \textit{Calcolo}, \textbf{37} (2000),
207--232.


\bibitem{Tang1993} T. Tang, A note on collocation methods for Volterra
integro-differential equuations with weakly singular kernels, \textit{IMA J. Numer. Anal.}, \textbf{13} (1993), 93-99.

\bibitem{28} T. Tang, X. Xu and J. Chen, On spectral methods for Volterra integral equations and the convergence analysis,
\textit{J. Comput. Math.}, \textbf{26} (2008), 825--837.

\bibitem{29} Y. Wei and Y. Chen, Convergence analysis of the spectral methods for weakly singular Volterra integro-differential
equations with smooth solutions, \textit{Adv. Appl. Math. Mech.}, \textbf{4} (2012), 1--20.


\bibitem{Thomee2006} V. Thom{\'e}e, Galerkin Finite Element Methods for
Parabolic Problems, Springer Ser. Comput. Math. 25, Springer-Verlag, Berlin, 2006.
\end{thebibliography}

\end{document}